\documentclass[11pt]{article}

\usepackage{amsmath,amssymb,amsthm}
\usepackage{graphicx}
\usepackage{subfigure}
\usepackage{enumerate}
\usepackage{fullpage}
\usepackage{url}
\usepackage{float}
\usepackage{xspace}
\usepackage{hyperref}
\usepackage{color}
\usepackage{thmtools, thm-restate}

\newcommand{\R}{\mathbb{R}}

\newcommand{\N}{\mathbb{N}}
\newtheorem{theorem}{Theorem}

\newtheorem{corollary}{Corollary}
\newtheorem{lem}{Lemma}[section]
\newtheorem{proposition}{Proposition}

\newtheorem{remark}{Remark}

\title{Upper bounds for $s$-distance sets and equiangular lines}

\author{Alexey Glazyrin\thanks{School of Mathematical \& Statistical Sciences,
         The University of Texas Rio Grande Valley, USA.}\ \ and Wei-Hsuan Yu\thanks{Mathematics Department, Michigan State University, USA.}}

\date{\today}

\begin{document}

\maketitle
\begin{abstract}

The set of points in a metric space is called an $s$-distance set if pairwise distances between these points admit only $s$ distinct values. Two-distance spherical sets with the set of scalar products $\{\alpha, -\alpha\}$, $\alpha\in[0,1)$, are called equiangular. The problem of determining the maximum size of $s$-distance sets in various spaces has a long history in mathematics. We suggest a new method of bounding the size of an $s$-distance set in compact two-point homogeneous spaces via zonal spherical functions. This method allows us to prove that the maximum size of a spherical two-distance set in $\mathbb{R}^n$, $n\geq 7$, is $\frac{n(n+1)}2$ with possible exceptions for some $n=(2k+1)^2-3$, $k \in \mathbb{N}$. We also prove the universal upper bound $\sim \frac 2 3 n a^2$ for equiangular sets with $\alpha=\frac 1 a$  and, employing this bound, prove a new upper bound on the size of equiangular sets in all dimensions. Finally, we classify all equiangular sets reaching this new bound.
\end{abstract}

\section{Introduction}

An $s$-distance set in a metric space $M$ is a finite set of points with exactly $s$ distinct pairwise distances. A number of classical problems can be formulated in terms of finding $s$-distance sets of maximum size. For example, the Ray-Chaudhuri--Wilson theorem \cite{ray75} gives an upper bound of $\binom{n}{s}$ for $s$-distance sets in the Johnson space $J^{n,w}$ (the space of subsets of an $n$-element set with exactly $w$ elements). A special case of this problem corresponds to the celebrated Erd\H{o}s-Ko-Rado theorem \cite{erd61,wil84}. There is extensive literature devoted to generalizations and improvements of these results \cite{dez78, dez83, fra81, fra86, alo91, bab95, ahl96, fu99, tan12, god15}.

In the Euclidean case $\mathbb{R}^n$, problems of finding upper bounds on the maximum size of $s$-distance sets are generally of two types: the number of distinct distances $s$ is very large compared to $n$, and the number of distinct distances is comparable to $n$. For the former case, the most well-studied problem is the famous Erd\H{o}s distinct distances problem for the plane \cite{erd46} which was essentially solved by Guth and Katz \cite{gut15} who proved that the size of an $s$-distance set on the plane is $O(s \ln s)$. In the latter case, the upper bound $\binom{n+s}{s}$ was given in \cite{ban83}. There is a natural example for each $s\leq n$ with the same asymptotic order as this bound: the set of $\binom {n+1}{s}$ centers of all $s$-faces of a regular simplex. Distances between the centers are defined only by how many common vertices their faces have. The study of two-distance sets was initiated in \cite{ein66}. Apart from the natural example described above, very little is known about lower bounds \cite{lis97}.

The above example works in the spherical case as well since all the face centers are equidistant from the center of the circumsphere. The first upper bounds for the maximum size of an $s$-distance set on the unit sphere $\mathbb{S}^{n-1}$ were found in \cite{del77} (the so-called \textit{harmonic bounds}, see Theorem \ref{thm:s-dist}). If an $s$-distance set attains the harmonic bound, then it forms a so-called tight spherical $2s$-design (see the survey paper \cite{ban09}). Several maximum spherical $s$-distance sets are known to form beautiful and important configurations on the unit sphere. For instance, the maximum $3$-distance set in $\R^3$ is the icosahedron. The maximum $3$-distance sets in $\R^8$ is coming from the $E_8$ root system. It provides the solution to the kissing number problem in $\R^8$ \cite{lev79,odl79} and also, as it has been recently shown, generates an optimal sphere packing in $\R^8$ \cite{via16}. Maximum spherical $s$-distance sets also usually offer a solution for energy minimization problems on sphere. The table in \cite{coh07} of \textit{universal optimal codes}, minimizers of a large class of energy functions on the sphere, shows that most of them are maximum spherical $s$-distance sets.

Denote by $g(n)$ the maximum size of a two-distance set in $\mathbb{S}^{n-1}$. The harmonic bound gives $g(n)\leq \frac {n(n+3)}{2}$. This bound is attained on the regular pentagon for $n=2$. The results of \cite{del77} also imply that the harmonic bound can be reached for $n\geq 3$ only if $n=a^2-3$, where $a$ is an odd number. In fact for $a=3$ and $5$ two-distance sets with $27$ and $275$ points respectively do exist \cite{lin66,con69}. It follows from \cite{ban04,neb12} that this is not the case for infinitely many other values of $a$: $a$ cannot be equal to $7, 9, 13, 21, 45, 61, 69, 77,$~etc. In \cite{mus09} it was shown that $g(n)=\frac {n(n+1)}{2}$ for $7\leq n\leq 39$ except for $n=22,23$. The set of the values of $n$, where the natural lower bound is in fact optimal, was extended to almost all $n\leq 93$ in \cite{bar13} and almost all $n\leq 417$ in \cite{yu16f}.

As one of the main results of this paper we show that this pattern holds for all $n$.

\begin{restatable}{theorem}{twodist} \label{thm:main}
$$g(n)= \frac {n(n+1)}2 \text{ for all $n\geq 7$ with possible exceptions for some $n = (2k+1)^2-3$, $k \in \N$.}$$
\end{restatable} 

Finding maximum $s$-distance sets is sometimes equivalent to constructing tight spherical $t$-designs for $t=2s$ \cite{del77}. For $s=2$, the only unknown cases in Theorem \ref{thm:main} are exactly the long-standing open problems for the classification of tight spherical $4$-designs.

A particular case of two-distance spherical sets, namely equiangular sets, attracts a lot of interest as it arises in various areas of mathematics. For instance, there is an extensive literature in frame theory devoted to equiangular tight frames \cite{str03, hol04, sus07}.

By an equiangular spherical set we will mean a two-distance spherical set with scalar products $\alpha$ and $-\alpha$. The natural question is to determine the maximum cardinality $M(n)$ of an equiangular set in $\mathbb{S}^{n-1}$. This problem was posed in \cite{lin66} where a few values of $M(n)$ were determined. The general upper bound $\binom{n+1}{2}$ is due to Gerzon \cite{lem73} (see Theorem \ref{thm:gerzon}).

It is not hard to show the connection between equiangular sets attaining Gerzon's bound and two-distance sets reaching the harmonic bound. For each equiangular set with $\frac {n(n+1)}{2}$ points it is possible to construct a derived set (see Subsection \ref{sub:gegen} for more details) in $\mathbb{S}^{n-2}$ with $\frac {n(n+1)}{2}-1=\frac {(n-1)(n+2)}{2}$ points that must attain the harmonic bound.

It is much more difficult to construct equiangular sets of large size. The first $O(n^2)$ construction was found by de Caen in \cite{cae00} (see also \cite{jed15,gre16} for large equiangular sets). This construction has $\frac 2 9 (n+1)^2$ points in $\mathbb{S}^{n-1}$, where $n=3\cdot 2^{2t-1}$, $t\in\mathbb{N}$. There are also $O(n^{3/2})$ constructions of equiangular sets built on Taylor graphs and projective planes \cite{lem73}.

We improve Gerzon's general bound on equiangular sets for almost all natural $n$.

\begin{restatable}{theorem}{equiang}\label{thm:equiang}
If $n \geq 359$, then 
$$M(n) \leq \frac{(a^2-2)(a^2-1)}2,
$$ 
where $a$ is the unique positive odd integer such that $a^2-2 \leq n \leq (a+2)^2-3$.
\end{restatable}

If $n=a^2-2$ for some odd integer $a$ the upper bound of Theorem \ref{thm:equiang} coincides with Gerzon's bound. For all other $n\geq 359$, the bound of Theorem \ref{thm:equiang} strictly improves Gerzon's bound.

Denote by $M_{\alpha} (n)$ the cardinality of the largest equiangular set in $\mathbb{S}^{n-1}$ with scalar products $\alpha,-\alpha$. The question of finding maximum equiangular sets with prescribed scalar products was raised in \cite{lem73}. In their paper, Lemmens and Seidel found that $M_{\frac 1 3} (n)\leq 2n-2$ for $n\geq 15$. They also proved that if $n\alpha^2<1$, then $M_{\alpha}(n)\leq\frac {n(1-\alpha^2)}{1-n\alpha^2}$ (the so-called \textit{relative bound}, see Theorem \ref{thm:rel}). Generally, it was shown by Neumann (see \cite{lem73}) that $M_{\frac 1 a} (n)\leq 2n$ unless $a$ is an odd natural number. For a fixed odd natural $a$ the behavior of $M_{\frac 1 a} (n)$ was not known before the paper of Bukh \cite{buk16}, where he showed that $M_{\frac 1 a} (n) \leq cn$, where $c= 2^{O(a^2)}$. Recently this bound was significantly improved by Balla, Dr\"{a}xler, Keevash, and Sudakov. In \cite{bal16} they showed that for sufficiently large $n$, $M_{\frac 1 a} (n)\leq 2n-2$. If $a\neq 3$, their bound is even better, namely $M_{\frac 1 a}(n)\leq 1.93n$.

We prove a new universal bound on $M_{\frac 1 a} (n)$ that works for all $a\geq 3$ and all $n$.

\begin{restatable}{theorem}{secondbound} \label{thm:bound2}
$$M_{\frac 1 a}(n)  \leq n\left(\frac 2 3 a^2 + \frac {4} {7}\right)+2,  \text{ for all $a\geq 3$ and for all $n \in \N$.}$$
\end{restatable}

The main benefit of this bound is its universality. Only for $a^2 \gtrsim \frac 3 4  n$, Gerzon's bound is a better upper bound for $M_{\frac 1 a} (n)$. For a fixed $a$, the bound of Theorem \ref{thm:bound2} is asymptotically inferior to the bound from \cite{bal16}. However, the bound from \cite{bal16} is valid only for sufficiently large $n$ (depending on $a$) and our bound works for all pairs $(a,n)$. We also note that the bound of Theorem \ref{thm:bound2} is usually better than SDP bounds in \cite{bar14} and \cite{kin16} when $a$ is relatively small compared to $n$ (see Section \ref{sect:2-dist} for details). 

Finally, we classify all extremal cases, where the bound of Theorem \ref{thm:equiang} is attained.

\begin{restatable}{theorem}{extremal}\label{thm:extr1}
For each odd natural number $a\geq 3$, if $n\leq 3a^2-16$, then the equiangular set $X\subset\mathbb{S}^{n-1}$ with scalar products $\frac 1 a, -\frac 1 a$ and exactly $\frac{(a^2-2)(a^2-1)}2$ points must belong to an $(a^2-2)$-dimensional subspace.
\end{restatable}

\begin{restatable}{corollary}{extremaltwo}\label{cor:extr2}
Let $n \geq 359$ and let $a$ be the unique positive odd integer such that $a^2-2 \leq n \leq (a+2)^2-3$. Suppose that $X\subset\mathbb{S}^{n-1}$ is an equiangular set of size $\frac{(a^2-2)(a^2-1)}2$. Then $X$ has scalar products $\frac 1 a,-\frac 1 a$ and is contained in an $(a^2-2)$-dimensional subspace.
\end{restatable}

The paper is organized as follows. In Section \ref{sect:s-dist} we overview the main methods applicable to the problem of finding maximal $s$-distance sets. Subsection \ref{sub:zon} gives a brief introduction to zonal spherical functions on compact two-point homogeneous spaces. In Subsection \ref{sub:gegen}, we show how Gegenbauer polynomials can be used to prove the relative bound on equiangular sets by Lemmens and Seidel \cite{lem73} and the bound from \cite{yu16f}. In \ref{sub:pol} we describe the polynomial method and demonstrate by giving short proofs of the bounds from \cite{del77,lem73}. Section \ref{sect:general} is devoted to a new method of finding upper bounds on $s$-distance sets in two-point homogeneus spaces. In Section \ref{sect:2-dist} we show how this method can be applied to two-distance spherical sets. Subsection \ref{sub:sphere} contains the general spherical bounds and in Subsection \ref{sub:equiang} we derive new bounds on equiangular sets. Proofs of Theorems \ref{thm:main} and \ref{thm:equiang} are contained in Section \ref{sect:proofs}. In Section \ref{sect:extremal}, we classify all extremal equiangular sets attaining the bound from Theorem \ref{thm:equiang}. Section \ref{sect:discuss} is devoted to open questions and discussion.

\section{$s$-distance sets}\label{sect:s-dist}

In this section we overview two general approaches used for finding upper bounds on $s$-distance sets. 

\subsection{Zonal spherical functions in compact two-point homogeneous spaces}\label{sub:zon}

The approach using zonal spherical functions for finding packing bounds was introduced by Delsarte et al \cite{del73,del77} and then generalized by Kabatyansky and Levenshtein \cite{kab78} to all compact two-point homogeneous spaces. Here we briefly define zonal spherical functions and explain how they can be used in this context.

A metric space $M$ with group $G$ acting on it is called two-point homogeneous if $G$ acts transitively on ordered pairs with the same distance. Infinite two-point homogeneous spaces with the structure of a connected Riemannian manifold are classified by Wang \cite{wan52} and include only spheres; real, complex, quaternionic spaces; and the Cayley projective plane.

Suppose $G$ is a connected compact group or a finite group. Taking $H$ as a stabilizer of some fixed point $x_0$, we can associate each point in $M$ with the corresponding left coset $gH$. Let $\mu$ be the Haar measure for $G$. This measure induces the unique invariant measure on $M$ which we also denote by $\mu$.

By $L^2(G)$ we denote the vector space of complex functions $f$ on $G$ satisfying

$$\int\limits_G |f(g)|^2d\mu (g)<\infty$$

with the standard inner product 

$$(f_1,f_2)=\int\limits_G f_1(g)\overline{f_2(g)}d\mu(g).$$

By taking only functions of $L^2(G)$ constant on left cosets of $H$, we can define $L^2(M)$ with the analogously defined inner product

$$(f_1,f_2)=\int\limits_M f_1(x)\overline{f_2(x)}d\mu(x).$$

By the Peter-Weyl theorem \cite{pet27}, $L^2(G)$ splits into an orthogonal direct sum of irreducible finite-dimensional representations of $G$. With certain conditions on $G$ (see, for instance, \cite{kab78} or \cite[Chapter 9]{con88} for details) this decomposition induces a decomposition of $L^2(M)$ into mutually orthogonal subspaces $\{V^{(k)}, k=0,1,\ldots\}$, $h_k=dim\ V^{(k)}$, where each space defines an irreducible representation of $G$. For an orthonormal basis $\{e_1^{(k)},\ldots,e_{h_k}^{(k)}\}$ of $V^{(k)}$, define for each $k$

$$P_k(x,y)=\sum\limits_{i=1}^{h_k} e_i^{(k)}(x)\overline{e_i^{(k)}(y)}.$$

Since $V^{(k)}$ defines a unitary representation of $G$, $P_k(gx,gy)=P_k(x,y)$ and, because of two-point homogeneity of $M$, $P_k(x,y)$ depends only on the distance between $x$ and $y$. The function $P_k$ is called \textit{a zonal spherical function} associated with $V^{(k)}$.

In practice, it is often convenient to use a modified distance function $\tau(x,y)$ (we call it just a distance function from now on) to express $P_k(x,y)$ as a function of one variable $P_k(\tau(x,y))$. Throughout the paper we will also use the notation $\tau_0=\tau(x,x)$.

The definition of $P_k$ implies that, for any finite set of points $\{x_1,\ldots, x_N\}$ and any set of complex numbers $\{a_1,\ldots, a_N\}$,

$$\sum\limits_{i=1}^N \sum_{j=1}^N P_k(x_i,x_j)a_i\overline{a_j}\geq 0$$ so the following proposition is true.

\begin{proposition}\label{prop:psd}
For any finite set $X=\{x_1,\ldots, x_N\}$ and any $k\geq 0$, the matrix $\left(P_k(\tau(x_i,x_j))\right)$ is positive semidefinite.
\end{proposition}

For the spherical case, the standard scalar product $(x,y)$ is used as a distance function $\tau(x,y)$ and the zonal spherical functions are given by Gegenbauer polynomials which can be defined recursively as follows:

$$G_0^{(n)}(t)=1,\ \ \ G_1^{(n)}(t)=t,$$
$$G_k^{(n)}(t)=\frac {(n+2k-4)tG_{k-1}^{(n)}(t)-(k-1)G_{k-2}^{(n)}(t)}{n+k-3}.$$

Here for convenience we use the normalized version of Gegenbauer polynomials, i.e. for all of them, $G_k^{(n)}(1)=1$.

Since it will be used further in the paper, we would like to write down Gegenbauer polynomials of degree 2 and 3 as well:

$$G_2^{(n)}(t)=\frac {nt^2-1}{n-1},\ \ \ G_3^{(n)}(t)=\frac {(n+2)t^3-3t}{n-1}.$$

Schoenberg \cite{sch42} proved the converse version of Proposition \ref{prop:psd} for the spherical case by showing that the real function $f$ satisfies $\left(f(\tau(x_i,x_j))\right)\succeq 0$ for any finite set $X=\{x_1,\ldots, x_N\}$ only if $f$ is a non-negative combination of Gegenbauer polynomials. Bochner \cite{boc41} proved the general version of this theorem for zonal spherical functions.

\subsection{Bounds on equiangular sets via Gegenbauer polynomials}\label{sub:gegen}

Here we show how to use zonal spherical functions to prove upper bounds on the size of equiangular sets. The following relative bound was proved by Lemmens and Seidel \cite{lem73}.

\begin{theorem}[Relative bound]\label{thm:rel}
For an $n$-dimensional equiangular set $X$ with scalar products $\{\alpha,-\alpha\}$ with $n\alpha^2<1$,

$$|X|\leq \frac {n(1-\alpha^2)}{1-n\alpha^2}$$
\end{theorem}

\begin{proof}
Let $X=\{x_1,\ldots, x_N\}$. The matrix $\left(G_2^{(n)}(x_i,x_j)\right)$ is positive semidefinite by Proposition \ref{prop:psd} so its sum of elements must be non-negative. Note that its diagonal elements are 1 and all the non-diagonal elements are $\frac{n\alpha^2-1}{n-1}$ so we get $N+N(N-1)\frac{n\alpha^2-1}{n-1}\geq 0$ and, therefore, $N\leq \frac {n-1}{1-n\alpha^2} +1 =  \frac {n(1-\alpha^2)}{1-n\alpha^2}.$
\end{proof}

As one more example we show how to obtain a short proof of the main theorem from \cite{yu16f}. The proof in \cite{yu16f} relied on a more intricate approach of Bachoc and Vallentin \cite{bac08a}.

\begin{theorem}\label{thm:2}
$$M_{\frac 1 a}(n) \leq  \frac {(a^2-2)(a^2-1)} 2  $$
for all $n$ and $a$ such that $n \leq 3a^2-16$ and $a \geq 3$.
\end{theorem}

Before proving this theorem we describe the construction of \textit{derived spherical $s$-distance sets}.  For an $s$-distance set $Z\subset\mathbb{S}^{n-1}$ the set of all points of $Z$ equidistant from some fixed point $x\in Z$ form an $s$-distance set with the same set of distances in a sphere of smaller dimension and smaller radius. By shifting these points and dilating the sphere we can get an $s$-distance set on a unit sphere. Defining this formally, the set $Z_{x,\alpha}$ derived from $Z$, with respect to $x\in Z$ and scalar product $\alpha$, is $\left\{\frac{y-\alpha x}{\sqrt{1-\alpha^2}}|y\in Z, (y,x)=\alpha\right\}$. If the set of scalar products of $Z$ is $\{\alpha, \beta,\gamma,\ldots\}$, then $Z_{x,\alpha}$ will have scalar products $\{\frac {\alpha-\alpha^2}{1-\alpha^2}, \frac {\beta-\alpha^2}{1-\alpha^2}, \frac {\gamma-\alpha^2}{1-\alpha^2},\ldots\}$.

\begin{proof}[Proof of Theorem \ref{thm:2}]

Unless $a$ is an odd natural number (see \cite{lem73} and the explanation after Lemma \ref{lem:eig}), $M_{\frac 1 a}(n) \leq 2n\leq 2(3a^2-16)$ which is always smaller than $\frac {(a^2-2)(a^2-1)} 2$. Therefore, we are interested only in the case when $a$ is an odd natural. It is sufficient to prove the statement of the theorem for $n$ equal to $3a^2-16$ exactly since any set of smaller dimension will be realizable in larger dimensions as well. From this moment on, $n=3a^2-16$.

Consider an equiangular set $X$ with scalar products $\{\frac 1 a, -\frac 1 a\}$ in $\mathbb{S}^{n-1}$ and choose an arbitrary point $x_0\in X$. Note that any point in an equiangular set may be changed to its opposite and the set still remains equiangular. Using this procedure we make $(y,x_0)=\frac 1 a$ hold for all $y\in X$, $y\neq x_0$. Denote the equiangular set obtained this way by $Z$. Now we consider the derived set $Z_{x_0,\frac 1 a}$. This is a two-distance set in $\mathbb{S}^{n-2}$ with scalar products $\frac {1}{a+1}$, $\frac {-1}{a-1}$ and $|X|-1$ points.

Denote $|X|-1$ by $N$ and assume that among ordered pairs of points in $Z_{x_0,\frac 1 a}$ there are $N_1$ with scalar product $\frac {1}{a+1}$ and $N_2$ with scalar product $\frac {-1}{a-1}$. These values should satisfy $N(N-1)=N_1+N_2$.

By taking the sum of elements from matrices formed by values of Gegenbauer polynomials and applying Proposition \ref{prop:psd}, we get that $\sum G_k^{(n-1)}((x_i,x_j))\geq 0$, where the sum is taken over all ordered pairs of points in $Z_{x_0,\frac 1 a}$. Taking $k=1,3$ we get the following two inequalities

\begin{align}
N+ \frac 1 {a+1} N_1 + \frac {-1} {a-1} N_2 \geq 0 \\
N+ \frac{-2a-6}{(a^2-6)(a+1)^3} N_1 +\frac{-2a+6}{(a^2-6)(a-1)^3}N_2 \geq 0 
\end{align}

Taking the first inequality with a non-negative coefficient $\frac{16}{(a^2-6)(a+1)^2(a-1)^2}$ and adding to the second one, we get 
$$\frac{16}{(a^2-6)(a+1)^2(a-1)^2} \left(N+ \frac 1 {a+1} N_1 +\frac {-1} {a-1} N_2\right)+$$

$$+\left(N+ \frac{-2a-6}{(a^2-6)(a+1)^3} N_1 +\frac{-2a+6}{(a^2-6)(a-1)^3}N_2\right) \geq 0.$$

$$N\left(\frac{16}{(a^2-6)(a+1)^2(a-1)^2}+1\right)+(N_1+N_2)\frac {-2a^2+10}{(a^2-6)(a+1)^2(a-1)^2}\geq 0.$$

Since $N_1+N_2=N(N-1)$ we obtain

$$\frac{16}{(a^2-6)(a+1)^2(a-1)^2}+1\geq (N-1)\frac {2a^2-10}{(a^2-6)(a+1)^2(a-1)^2}$$

$$N\leq 1+\frac {16+(a^2-6)(a+1)^2(a-1)^2}{2a^2-10}=1+\frac {(a^2-5)(a^4-3a^2-2)}{2(a^2-5)}=\frac {a^4-3a^2}{2}.$$

$|X|=N+1\leq \frac {a^4-3a^2}{2} +1 =\frac {(a^2-2)(a^2-1)}{2}.$
\end{proof}

\begin{remark}
This proof allows us to get more information on extremal configurations than the proof in \cite{yu16f}. The upper bound is attained only if both inequalities become equalities and, subsequently, values of $N_1$ and $N_2$ are known exactly and must be the same regardless of the choice of the point $x_0$. We will use this to prove Theorem \ref{thm:extr1} in Section \ref{sect:extremal}.
\end{remark}

\subsection{Polynomial method}\label{sub:pol}

One of the key methods in analyzing sets with few distances is the polynomial method. The main idea is to find a set of polynomials nullifying almost all values of the distance function and use the following standard lemma.

\begin{lem}\label{lem:lin}
For an arbitrary set $X$ consider a vector space $\mathcal{V}$ of functions $f:X\rightarrow F$, where $F$ is a field. Assume that $f_1,\ldots,f_m\in V$, a finite-dimensional vector subspace of $\mathcal{V}$. Then for any finite set of elements $\{x_1,\ldots, x_N\}\subseteq X$, the rank of the matrix $\left(f_i(x_j)\right)$ is at most the dimension of $V$.
\end{lem}

\begin{proof}
Any linear dependence of the functions $f_i$ is also a linear dependence of the columns of this matrix, so the rank is not greater than the dimension of the space of functions.
\end{proof}

The straightforward application of this method in \cite{del77} allows one to prove a general bound on spherical $s$-distance sets (the so-called \textit{absolute bound}).

\begin{theorem}\label{thm:s-dist}
The size of an $s$-distance set in $\mathbb{S}^{n-1}$ is not greater than $\binom{n+s-1}{s}+\binom{n+s-2}{s-1}$
\end{theorem}

\begin{proof}
Suppose $X=\{x_1,\ldots,x_N\}$ is an $s$-distance set in $\mathbb{S}^{n-1}$ with the set of scalar products $\{\beta_1,\ldots,\beta_s\}$. Define polynomials $f_j(x)=\prod_{i=1}^s((x,x_j)-\beta_i)$. Matrix $\left(f_i(x_j)\right)$ has non-zero entries on its diagonal and 0 everywhere else. Hence its rank is $N$. By Lemma \ref{lem:lin} the rank is no greater than the dimension of the space of spherical polynomials with degree $\leq s$ which is exactly $\binom{n+s-1}{s}+\binom{n+s-2}{s-1}$.
\end{proof}

Similarly, we can get the bound on equiangular sets from \cite{lem73}.

\begin{theorem}[Gerzon's bound]\label{thm:gerzon}
The size of an equiangular set in $\mathbb{S}^{n-1}$ is not greater than $\frac {n(n+1)}{2}$.
\end{theorem}

\begin{proof}
Assume $X=\{x_1,\ldots,x_N\}$ is an equiangular set in $\mathbb{S}^{n-1}$ with scalar products $\pm\alpha$. For each $j$, define $f_j(x)=(x,x_j)^2-\alpha^2$. As in the proof of the previous theorem, the matrix $\left(f_i(x_j)\right)$ has non-zero entries on its diagonal and 0 everywhere else. The rank is $N$ and by Lemma \ref{lem:lin} it's not greater than the dimension of the space of polynomials generated by all monomials of degree 2 or 0 which is exactly $\binom{n+1}{2}$.
\end{proof}

Musin \cite{mus09} showed that the linear span of polynomials $((x,x_j)-\alpha)((x,x_j)-\beta)$ does not contain any monomials of degree 1 if $\alpha+\beta\geq 0$ and, subsequently, generalized Gerzon's bound to all two-distance sets with scalar products $\alpha, \beta$ such that $\alpha+\beta\geq 0$.

\begin{lem}\label{lem:musin}
The size of a two-distance set in $\mathbb{S}^{n-1}$ with scalar products $\alpha,\beta$ such that $\alpha+\beta\geq 0$ is not greater than $\frac {n(n+1)}{2}$.
\end{lem}

For the next lemma, suppose $X=\{x_1,\ldots, x_N\}$ is an $s$-distance set in a metric space $M$. $\tau$ is a distance function and $\tau(x,x)=\tau_0$ for all $x\in M$. Assume $\tau$ admits only real values $\beta_1,\ldots,\beta_s$ over all ordered pairs of points from $X$. Let $V_{s-1}$ be a linear space of real functions spanned by all $f_{Q,\xi}(x)=Q(\tau(x,\xi))$, where $\xi$ is a fixed point from $M$ and $Q$ is a real polynomial of degree $s-1$. Let $dim\ V_{s-1}=K$.

For each $l$, $1\leq l\leq s$, we can define the adjacency matrix $\Phi_l$, where the $ij$-entry of this matrix is 1 if $\tau(x_i,x_j)=\beta_l$ and 0 otherwise. Denote $k_l=-\prod_{i\neq l} \frac{\tau_0-\beta_i}{\beta_l-\beta_i}$.

The following lemma was essentially proved by Nozaki \cite{noz11} as the main ingredient in his proof of the generalization of \cite{lar77}.

\begin{lem}\label{lem:eig}
If the size of $X$ is at least $K$, then for each $l$, $1\leq l\leq s$, $k_l$ is an eigenvalue of the adjacency matrix $\Phi_l$ with multiplicity at least $|X|-K$.
\end{lem}

\begin{proof}
Define functions $f_j^l(x)=\prod_{i\neq l}\frac {\tau(x,x_j)-\beta_i}{\beta_l-\beta_i}$. $f_j^l(x_i)$ may take only three different values: if $i=j$ then $f_i^l(x_i)=\prod_{i\neq l}\frac {\tau_0-\beta_i}{\beta_l-\beta_i}=-k_l$; if $i\neq j$ and $\tau(x_i,x_j)=\beta_l$ then $f_i^l(x_i)=1$;  if $i\neq j$ and $\tau(x_i,x_j)\neq\beta_l$ then $f_i^l(x_i)=0$. Hence, for a fixed $l$, matrix $\left(f_i^l(x_i)\right)$ is $-k_l I + \Phi_l$. By Lemma \ref{lem:lin} the rank of this matrix is not greater than $K$ and, therefore, 0 is the eigenvalue of $-k_l I + \Phi_l$ with multiplicity at least $|X|-K$. From here we conclude that $\Phi_l$ has eigenvalue $k_l$ with multiplicity at least $|X|-K$.
\end{proof}

Using this lemma, it is not hard to prove that $M_{\frac 1 a} (n)>2n+2$ implies that $a$ is an odd natural number. Dimension $K$ in this scenario is $n+1$ and the eigenvalue $\frac {1+\frac 1 a}{\frac 1 a +\frac 1 a}=\frac {a+1}{2}$ will be a root of multiplicity $\geq |X|-(n+1)>\frac 1 2 |X|$ of the characteristic polynomial of the adjacency matrix $\Phi$. This polynomial has degree $|X|$ and integer coefficients. Any algebraic conjugate of $\frac {a+1}{2}$ must be a root of this polynomial with the same multiplicity but it would make the total degree greater than $|X|$. Hence there are no algebraically conjugate numbers and $\frac {a+1}{2}$ is integer. Proving the same condition with the assumption $M_{\frac 1 a} (n)> 2n$ as in \cite{lem73} takes a little more effort.

The numbers $k_l$ from Lemma \ref{lem:eig} satisfy the following useful properties (see \cite[Theorem 2.6]{mus11}).

\begin{lem}\label{lem:interp}
For any real polynomial $P$ of degree $\leq s-1$, $P(\tau_0)+\sum_{l=1}^s k_l P(\beta_l)= 0$.
\end{lem}

\begin{proof}
The polynomial $P(x)-\sum_{l=1}^s \prod_{i\neq l}\frac {x-\beta_i}{\beta_l-\beta_i} P(\beta_l)$ is 0 at all $\beta_i$. Since $P$ has $s$ roots, it must be identically 0 and, in particular, must be 0 at $\tau_0$.
\end{proof}

\begin{lem}\label{lem:interp_s}
$\tau_0^s+\sum_{l=1}^s k_l \beta_l^s= \prod_{l=1}^s (\tau_0-\beta_l)$.
\end{lem}

\begin{proof}
The polynomial $x^s-\sum_{l=1}^s \prod_{i\neq l}\frac {x-\beta_i}{\beta_l-\beta_i} \beta_l^s$ is 0 at all $\beta_i$. Since this polynomial has $s$ roots and its leading coefficient is 1, it must be the product of all $x-\beta_l$. Particularly, at $\tau_0$ it must be the product of all $\tau_0-\beta_l$.
\end{proof}

\section{General bound for $s$-distance sets}\label{sect:general}

In this section we derive a new general method to find upper bounds on the size of $s$-distance sets in compact two-point homogeneous spaces. Our setup is similar to the one for Lemma \ref{lem:eig}. We assume that $X=\{x_1,\ldots, x_N\}$ is an $s$-distance set in a compact two-point homogeneous metric space $M$ equipped with a distance function $\tau(x,y)$, $\tau(x,x)=\tau_0$, and zonal spherical functions $P_k(t)$ as described in Subsection \ref{sub:zon}. Assume $\tau$ admits only values $\{\beta_1,\ldots,\beta_s\}$ on pairs of points from $X$. For each $l$, $\Phi_l$ is an adjacency matrix for $\beta_l$. 

\begin{theorem}\label{thm:general}
Let $X$ be an $s$-distance set in $M$, $s\geq 2$, and let the distance function $\tau$ admit only values $\{\beta_1,\ldots,\beta_s\}$ on pairs of points from $X$. Let $K$ be the dimension of a linear space of real functions spanned by all $f_{Q,\xi}(x)=Q(\tau(x,\xi))$, where $\xi$ is a fixed point from $M$ and $Q$ is a real polynomial of degree $s-1$. Denote  $k_l=-\prod_{i\neq l} \frac{\tau_0-\beta_i}{\beta_l-\beta_i}$. Assume $|X|>1+K(s-1)$ and let $P(t)$ be any non-negative linear combination of zonal spherical functions $P_k(t)$ of the compact two-point homogeneous space $M$. Then
$$0 \leq (|X|-1-K(s-1))(P(\tau_0)+P(\beta_1)k_1+\ldots+P(\beta_s)k_s)\leq |X|P(\tau_0).$$
\end{theorem}

\begin{proof}
By Lemma \ref{lem:eig}, $k_l$ is an eigenvalue of the adjacency matrix $\Phi_l$ for $\beta_l$ and its multiplicity is at least $|X|-K$. For each $l$ denote the eigenspace of $\Phi_l$ corresponding to $k_l$ by $S_l$ and the eigenspace of $J$ (matrix of all 1's) corresponding to 0 by $S_0$. We consider the intersection $S=S_0\cap S_1\ldots\cap S_{s-1}$. Its codimension is not greater than the sum of codimensions of $S_0$, $\ldots$, $S_{s-1}$, hence it's not greater than $1+K(s-1)$. Therefore, its dimension is at least $|X|-K(s-1)-1$. Any vector from $S$ is an eigenvector of $\Phi_i$, $1\leq i\leq s-1$, with the eigenvalue $k_i$, an eigenvector of $J$ with eigenvalue 0, and an eigenvector of $I$ with eigenvalue 1. From the condition on matrices $I+\Phi_1+\ldots+\Phi_s=J$, we get that any vector of $S$ is also an eigenvector of $\Phi_s$ with the eigenvalue $-1-k_1-\ldots -k_{s-1}$. By Lemma \ref{lem:interp} for degree 0, this eigenvalue is equal to $k_s$. Therefore, any vector from $S$ belongs to $S_s$ as well.

The Gram matrix of $X$ is $G=I+\beta_1\Phi_1+\ldots+\beta_s\Phi_s$. For a real polynomial $P$, denote

$$P(G):=P(\tau_0)I+P(\beta_1)\Phi_1+\ldots+P(\beta_s)\Phi_s.$$

Note that if vector $v$ is in $S$, then

$$P(G)v=(P(\tau_0)+P(\beta_1)k_1+\ldots+P(\beta_s)k_s)v.$$

Hence any vector $v\in S$ is an eigenvector of $P(G)$ with the eigenvalue $P(\tau_0)+P(\beta_1)k_1+\ldots+P(\beta_s)k_s$. Therefore, the multiplicity $m$ of this eigenvalue of $P(G)$ is at least $|X|-1-K(s-1)$.

If $P$ is any non-negative linear combination of zonal spherical functions, then $P(G)\succeq 0$. Then any eigenvalue of $P(G)$ is a non-negative real number which proves the first inequality in the statement of the theorem. The sum of eigenvalues of $P(G)$ equals $tr(P(G))=|X|P(\tau_0)$ and, since all eigenvalues are non-negative,

$$|X|P(\tau_0)\geq m(P(\tau_0)+P(\beta_1)k_1+\ldots+P(\beta_s)k_s)\geq$$

$$\geq (|X|-1-K(s-1))(P(\tau_0)+P(\beta_1)k_1+\ldots+P(\beta_s)k_s)$$ which proves the second inequality of the theorem.
\end{proof}

\begin{remark}\label{rem:regular}
By a regular $s$-distance set we mean an $s$-distance set such that the distribution of distances from any point of this set is the same. The result of Theorem \ref{thm:main} and all subsequent corollaries may be improved if the set is known to be regular. In this case, $S_1\ldots\cap S_{s-1}$ already is a subset of $S_0$ and the codimension will be no greater than $K(s-1)$. Numerators in all bounds of Corollaries \ref{cor:general}-\ref{cor:2-dist} will be smaller by one with this condition.
\end{remark}

For the next statement we assume that $P_s(t)$ is a real polynomial of degree $s$ and use it in Theorem \ref{thm:general}. Let the highest degree coefficient of $P_s(t)$ be $c_s$.

\begin{corollary}\label{cor:general}
With the notation used for Theorem \ref{thm:general},
$$|X|\leq \dfrac {1+K(s-1)}{1-\dfrac {P_s(\tau_0)}{c_s \prod_{l=1}^s (\tau_0-\beta_l)}},$$

if $|X|>1+K(s-1)$ and the right hand side of the inequality is positive.
\end{corollary}

\begin{proof}
We can write $P_s(t)=c_st^s+Q(t)$, where $Q$ has degree $\leq s-1$. Then

$$(P(\tau_0)+P(\beta_1)k_1+\ldots+P(\beta_s)k_s)=c_s(\tau_0^s+\sum_{l=1}^s k_l \beta_l^s)+(Q(\tau_0)+\sum_{l=1}^s k_l Q(\beta_l)).$$

By Lemmas \ref{lem:interp} and \ref{lem:interp_s}, this is equal to $c_s \prod_{l=1}^s (\tau_0-\beta_l)$. Due to the first inequality from Theorem \ref{thm:general} this value is positive. To finish the proof we transform the second inequality from Theorem \ref{thm:general}:

$$\left(1-\frac {1+K(s-1)}{|X|}\right)c_s \prod_{l=1}^s (\tau_0-\beta_l)\leq P(\tau_0),$$

$$1-\frac {P(\tau_0)}{c_s \prod_{l=1}^s (\tau_0-\beta_l)}\leq \frac {1+K(s-1)}{|X|},$$

$$|X|\leq \dfrac {1+K(s-1)}{1-\dfrac {P_s(\tau_0)}{c_s \prod_{l=1}^s (\tau_0-\beta_l)}}.$$

\end{proof}

\section{Upper bounds for two-distance sets}\label{sect:2-dist}

\subsection{General spherical bounds}\label{sub:sphere}

In this section we show how to obtain new upper bounds for spherical two-distance sets. Firstly, we rewrite Corollary \ref{cor:general} for the specific case of spherical sets.

\begin{corollary}\label{cor:spherical}
Let $X$ be a spherical $s$-distance set in $\mathbb{S}^{n-1}$ with scalar products $\beta_1,\ldots,\beta_s$. Then
$$|X|\leq \dfrac {1+\left(\binom{n+s-2}{s-1}+\binom{n+s-3}{s-2}\right)(s-1)}{1-\dfrac {1}{\prod_{i=2}^s \frac {n+2i-4}{n+i-3} \prod_{l=1}^s (1-\beta_l)}},$$

if the right hand side is positive.
\end{corollary}

\begin{proof}
For the spherical case $\tau(x,y)=(x,y)$ so $\tau_0=1$. The leading coefficient of the normalized Gegenbauer polynomial of degree $s$ is precisely $\prod_{i=2}^s \frac {n+2i-4}{n+i-3}$. Finally, the space of spherical polynomials of degree $\leq s-1$ has dimension $\binom{n+s-2}{s-1}+\binom{n+s-3}{s-2}$. The corollary is then obtained by substituting these values in Corollary \ref{cor:general}. In this case the denominator in the inequality is always less than 1 so it works even when $|X|\leq 1+\left(\binom{n+s-2}{s-1}+\binom{n+s-3}{s-2}\right)(s-1)$.
\end{proof}

By taking $s=2$ we get the result for two-distance sets.

\begin{corollary}\label{cor:2-dist}
Let $X$ be a spherical two-distance set in $\mathbb{S}^{n-1}$ with scalar products $\alpha,\beta$. Then
$$|X|\leq \dfrac {n+2}{1-\frac {n-1}{n(1-\alpha)(1-\beta)}},$$

if the right hand side is positive.
\end{corollary}

\subsection{Bounds for equiangular sets}\label{sub:equiang}

In this subsecton we assume that $X$ is an equiangular set in $\mathbb{S}^{n-1}$ with scalar products $\alpha=\frac 1 a,\beta=-\frac 1 a$. Using Corollary \ref{cor:2-dist} in a straightforward manner will not bring any new bounds for equiangular sets. In fact, we can only get $|X|\leq \frac {(n+2)n(1-\alpha^2)}{1-n\alpha^2}$, which is worse than the relative bound from Theorem \ref{thm:rel}. The trick here is to consider derived sets and apply the bound to them.

\begin{theorem}\label{thm:bound1}
For any $a\geq 3$, $$M_{\frac 1 a}(n)\leq \frac {n^2a^2+n-2}{n+a^2-2}<na^2.$$
\end{theorem}

\begin{proof}
Just like in the proof of Theorem \ref{thm:2}, we choose an arbitrary point $x_0\in X$ and change some points of $X$ to their opposites so that $(y,x_0)=\frac 1 a$ for all $y\in X$, $y\neq x_0$. We denote the resulting equiangular set by $Z$ and consider the derived set $Z_{x_0,\frac 1 a}$. This is a two-distance set in $\mathbb{S}^{n-2}$ with scalar products $\frac 1 {a+1}$ and $\frac {-1}{a-1}$. By Corollary \ref{cor:2-dist},

$$|Z_{x_0,\frac 1 a}|\leq \frac {n+1}{1-\frac {n-2}{(n-1)(1-\frac 1 {a+1})(1+\frac 1 {a-1})}}=\frac {n+1}{1-\frac {(n-2)(a^2-1)}{(n-1)a^2}}=\frac{(n^2-1)a^2}{n+a^2-2}.$$

$$|X|=|Z_{x_0,\frac 1 a}|+1\leq \frac{(n^2-1)a^2}{n+a^2-2}+1=\frac {n^2a^2+n-2}{n+a^2-2}<\frac {n^2a^2+na^2(a^2-2)}{n+a^2-2}=na^2.$$
\end{proof}

When $n\sim a^2$, the bound in this theorem is asymptotically the same as Gerzon's bound, i.e. $\sim \frac 1 2 n^2$. The reason why this bound cannot be precise for the cases when Gerzon's bound is attained lies in the proof of Theorem \ref{thm:general}. Configurations where Gerzon's bound is attained are known to be determined by regular graphs so, as we mentioned in Remark \ref{rem:regular}, we should use codimension $\leq n+1$ instead of $n+2$ which eventually ended up in the numerator of Corollary \ref{cor:2-dist}. Using this corollary with the additional condition of regularity would lead to the very same Gerzon bound for $n=a^2-2$.

When $n\gg a^2$, it makes sense to use the upper bound $na^2$. For the size of equiangular sets with prescribed scalar products, the first bound linear in $n$ was proved in \cite{buk16} and it was asymptotically $\leq cn$, where $c=2^{O(a^2)}$. The bound from \cite{bal16} is much better, $1.93n$ unless $a=3$ for which it was already known to be $2n-2$ \cite{lem73}, but this bound holds only for sufficiently large $n$. Essentially our bound fills the gap between Gerzon's bounds $\frac {n(n+1)}{2}$, when $n$ is really close to $a^2$, and $1.93n$, when $n$ is sufficiently large with respect to $a^2$.

The $na^2$ bound of Theorem \ref{thm:bound1} can be improved by a more careful analysis of derived sets.

\secondbound*

\begin{proof}

We consider the set $T:=Z_{\frac 1 a, x_0}$ from the proof of Theorem \ref{thm:2}. It is a two-distance set in $\mathbb{S}^{n-2}$ with $|X|-1$ points and scalar products $\frac 1 {a+1}$ and $\frac {-1}{a-1}$. Choose an arbitrary vertex $x_{00}$ in this set. The remaining vertices are split into two sets depending on the scalar product with respect to $x_{00}$. Consider the derived sets $T_1=T_{\frac 1 {a+1}, x_{00}}$ and $T_2=T_{\frac {-1} {a-1}, x_{00}}$.

$T_1$ is a two-distance set in $\mathbb{S}^{n-3}$ with scalar products $\frac 1 {a+2}$ and $-\frac {a+3}{(a-1)(a+2)}$. Applying Corollary \ref{cor:2-dist} to $T_1$, we get that

$$|T_1|\leq \frac {n}{1-\frac {n-3}{(n-2)(1-\frac 1 {a+2})(1+\frac {a+3}{(a-1)(a+2)})}}=\frac {n}{1-\frac {n-3}{(n-2)\left(1+\frac {3a+5}{(a-1)(a+2)^2}\right)}}=$$

\begin{equation}\label{eqn:3}
=\frac {n(n-2)\left(1+\frac {3a+5}{(a-1)(a+2)^2}\right)}{1+(n-2)\frac {3a+5}{(a-1)(a+2)^2}}=\frac {n\left(1+\frac {3a+5}{(a-1)(a+2)^2}\right)} {\frac 1 {n-2}+\frac {3a+5}{(a-1)(a+2)^2}},
\end{equation}

which is strictly smaller than

$$\frac {n\left(1+\frac {3a+5}{(a-1)(a+2)^2}\right)} {\frac {3a+5}{(a-1)(a+2)^2}}=n\left(\frac {(a-1)(a+2)^2}{3a+5} + 1\right).$$

$T_2$ is a two-distance set in $\mathbb{S}^{n-3}$ with scalar products $-\frac 1 {a-2}$ and $\frac {a-3}{(a+1)(a-2)}$. By Corollary \ref{cor:2-dist} for $T_2$ and calculations similar to those for $T_1$,

\begin{equation}\label{eqn:4}
|T_2|\leq \frac {n(n-2)\left(1+\frac {3a-5}{(a+1)(a-2)^2}\right)}{1+(n-2)\frac {3a-5}{(a+1)(a-2)^2}} = \frac {n\left(1+\frac {3a-5}{(a+1)(a-2)^2}\right)} {\frac 1 {n-2}+\frac {3a-5}{(a+1)(a-2)^2}},
\end{equation}

which is strictly smaller than

$$n\left(\frac {(a+1)(a-2)^2}{3a-5} + 1\right).$$

Combining these two inequalities and taking into account points $x_0$ and $x_{00}$, we get that

$$|X|\leq n\left( \frac {(a-1)(a+2)^2}{3a+5}+\frac {(a+1)(a-2)^2}{3a-5}+2 \right)+2\leq n\left(\frac 2 3 a^2+\frac 4 7\right)+2,$$

since $\frac {(a-1)(a+2)^2}{3a+5}+\frac {(a+1)(a-2)^2}{3a-5}+2-\frac 2 3 a^2$ is a strictly decreasing function for $a\geq 3$ and its value at 3 is $\frac 4 7$.

\end{proof}

The bound of Theorem \ref{thm:bound2} is usually better than SDP bounds in \cite{bar14} and \cite{kin16} when $a$ is relatively small compared to $n$. For instance, using inequalities (\ref{eqn:3}) and (\ref{eqn:4}) from the proof of Theorem \ref{thm:bound2} we can prove the upper bound of 2224 for $M_{\frac 1 5} (137)$ which is much better than the SDP bound of 9529 from \cite{bar14} and the bound of 6743 for $M_{\frac 1 5}(400)$ which is significantly better than 17595 from \cite{kin16}.

The choice of a new vertex and split of the set into two two-distance subsets of smaller dimension may be done any number of times. By iterating this process $l$ times, i.e. choosing $l$ arbitrary points in $T$ and dividing all other vertices into $2^l$ groups depending on their distances to the chosen vertices, we can get asymptotically ($n/a
\rightarrow\infty$) the upper bound $\sim \frac {2^l}{2l+1} n a^2$ which is optimal for $l=1$.

\section{Proof of Theorems \ref{thm:main} and \ref{thm:equiang}}\label{sect:proofs}

\equiang*

\begin{proof}
Since $359=19^2-2$, $a\geq 19$. Assume $M(n)=M_{\frac 1 b} (n)$. If $b$ is not an odd natural number, $M(n)\leq 2n\leq 2((a+2)^2-3)<\frac {(a^2-2)(a^2-1)}{2}$ for all suitable $a$. There are three possible cases for $b$.

\begin{enumerate}

\item $b^2-2\leq n \leq 3b^2-16$. In this case $b\leq a$ because $a$ is the largest odd integer such that $a^2-2\leq n$. By Theorem \ref{thm:2}, $$M_{\frac 1 b} (n)\leq \frac {(b^2-2)(b^2-1)}{2}\leq \frac {(a^2-2)(a^2-1)}{2}.$$

\item $n\leq b^2-3$. By Theorem \ref{thm:rel} (relative bound), $M_{\frac 1 b} (n)\leq \frac {n(b^2-1)} {b^2-n}$. In this case, $b\geq a+2$ and among all possible values $b=a+2$ maximizes $\frac {n(b^2-1)} {b^2-n}$ so $M_{\frac 1 b} (n)\leq \frac {n((a+2)^2-1)} {(a+2)^2-n}$. With respect to $n$, this value is maximal when $n=(a+2)^2-3$. Therefore, $$M_{\frac 1 b} (n)\leq\frac {((a+2)^2-3)((a+2)^2-1)} {3}<\frac {(a^2-2)(a^2-1)}{2}$$ for all $a\geq 19$.

\item $n\geq 3b^2-15$. By Theorem \ref{thm:bound2},

$$M_{\frac 1 b}(n)\leq n\left(\frac 2 3 b^2+\frac 4 7\right)+2\leq n\left(\frac 2 3 \cdot\frac {n+15}{3}+\frac 4 7\right)+2=\frac {14n^2+246n+126}{63}\leq$$

$$\leq \frac {14((a+2)^2-3)^2+246((a+2)^2-3)+126}{63}<\frac {(a^2-2)(a^2-1)}{2}$$ for all $a\geq 19$.

\end{enumerate}
\end{proof}

\twodist*

\begin{proof}
We will use the following observation made in \cite{del77}. If $X=\{x_1,\ldots,x_N\}$ is a spherical two-distance set in $\mathbb{S}^{n-1}$ with scalar products $\alpha$ and $\beta$ such that $\alpha+\beta<0$, then there is an equiangular set of the same size in $\mathbb{S}^n$. For the construction of this set consider a unit vector $y$ orthogonal to all $x_i$. Then for any $t\in[0,1]$, the set of points $tx_i+\sqrt{1-t^2}y$ is a set in $\mathbb{S}^n$ with only two scalar products $t^2\alpha+(1-t^2)$ and $t^2\beta+(1-t^2)$. The sum of these scalar products is $t^2(\alpha+\beta)+2(1-t^2)$, which is negative at $t=1$ and positive at $t=0$. Hence there is $t$, where this sum is 0 and the set is equiangular.

Therefore, if $X$ is a two-distance set with $\alpha+\beta<0$ in $\mathbb{S}^{n-1}$, its size cannot be larger than $M(n+1)$. For $n\geq 358$, using Theorem \ref{thm:equiang} we get that, unless $n+1=a^2-2$ for some odd $a$, $M(n+1)\leq \frac {n(n+1)}{2}$. The statement of the theorem is true for $n<359$ due to computations in \cite{kin16} and the same observation (see \cite{yu16f} for details).

Finally, the case when $\alpha+\beta\geq 0$ is covered for all dimensions by Lemma \ref{lem:musin}.
\end{proof}

\section{Extremal equiangular sets}\label{sect:extremal}

In this section we analyze what equiangular sets attain the bound of Theorem \ref{thm:equiang}. In order to prove Theorem \ref{thm:extr1} we will need the definition of a strongly regular graph and its basic properties.

By a strongly regular graph with parameters $(v,k,\lambda,\mu)$ we mean a regular simple graph with $v$ vertices and degree $k$ such that every two connected vertices have exactly $\lambda$ common neighbors and every two non-connected vertices have exactly $\mu$ common neighbors. If $\Phi$ is an incidence matrix of such strongly regular graph, then, from the definition of a strongly regular graph, $\Phi^2=kI+\lambda\Phi+\mu(J-I-\Phi)$. Vector {\bf 1} of all 1's is an eigenvector of $\Phi$ with eigenvalue $k$ so from this matrix equality multiplied by {\bf 1} we get that parameters are not independent: $k^2=k+\lambda k + \mu(v-1-k)$. All other eigenvectors of $\Phi$ must be orthogonal to {\bf 1} so they are annihilated by $J$ and all other eigenvalues must satisfy $e^2=k+\lambda e + \mu(-1-e)$. Unless $\mu=0$ (the case of a disjoint union of complete graphs with $k+1$ vertices each), $e$ cannot be the same as $k$ so the eigenspace for $k$ is one-dimensional. Two other eigenvalues $e_1$ and $e_2$ may be found from the quadratic equation above. Dimensions of their respective eigenspaces must satisfy equation $1+d_1+d_2=v$ for the total dimension and $k+d_1e_1+d_2e_2=0$ for the trace of $\Phi$ which can allow one to find the dimensions precisely.

\extremal*

\begin{proof}
We continue with the same notation as in the proof of Theorem \ref{thm:2}. $X$ is an equiangular set in $\mathbb{S}^{n-1}$ with scalar products $\{\frac 1 a, -\frac 1 a\}$. An arbitrary point $x_0\in X$ was chosen and some points of $X$ were switched to their opposite so that $(y,x_0)=\frac 1 a$ for all $y\in X$. The set obtained this way was called $Z$ and we considered the derived set $Z_{x_0,\frac 1 a}$. This is a two-distance set in $\mathbb{S}^{n-2}$ with scalar products $\frac {1}{a+1}$, $\frac {-1}{a-1}$ and $|X|-1$ points. We denoted $|X|-1$ by $N$ and assumed that among ordered pairs of points in $Z_{x_0,\frac 1 a}$ there are $N_1$ with scalar product $\frac {1}{a+1}$ and $N_2$ with scalar product $\frac {-1}{a-1}$.

Consider a graph $G$ on $N$ vertices in $Z_{\frac 1 a, x_0}$ such that two vertices $x$ and $y$ form an edge if and only if $(x,y)=\frac 1 {a+1}$. Our goal is to show that this graph is a strongly regular graph.

At the moment from the proof of Theorem \ref{thm:2} we know that $N=\frac {(a^2-3)a^2}{2}$ and inequalities used in the proof must become equalities since the bound is attained. Therefore, numbers $N_1$ (twice the number of edges in $G$) and $N_2$ (twice the number of non-edges in $G$) may be found precisely: $N_1=\frac {(a+1)N(N-a)}{2a}$, $N_2=N(N-1)-N_1$.

The main observation under the proof is that we can choose any vertex of the set $X$ for $x_0$. Depending on a vertex taken we can get different graphs. It follows from the proof of Theorem \ref{thm:2} that all these graphs must have the same number of edges.

First, let us analyze how the graph is changed when the initial vertex for switching is changed. Assume we take another point $x$, associated with vertex $u$ in $G$, and make a derived set using $x$ instead of $x_0$. Denote by $N(u)$ the set of neighbors of $u$ in $G$ and denote by $N'(u)$ the set of non-neighbors. In the equiangular set $Z$, the set of points having scalar product $\frac 1 a$ with $x$ consists of $x_0$ and points corresponding to $N(u)$ and the set of points having scalar product $-\frac 1 a$ with $x$ consists of points corresponding to $N'(u)$. Therefore, we need to switch all points for $N'(u)$ to their opposites. After this operation, $x_0$ has scalar products of $\frac 1 a$ with $x$ and points for $N(u)$ and has scalar products of $-\frac 1 a$ with points for $N'(u)$. All scalar products between $N(u)$ and $N'(u)$ changed their signs.

When constructing a new graph $G'$ we should delete the vertex $u$, introduce a new vertex $u'$ corresponding to the point $x$ which is connected to all vertices in $N(u)$ and not connected to any vertex in $N'(u)$, and swap edges and non-edges between $N(u)$ and $N'(u)$. Note that the new vertex $u'$ has the same adjacencies as the old vertex $u$ so essentially $G'$ is the graph $G$ in which edges and non-edges between $N(u)$ and $N'(u)$ were swapped.

Denote by $k$ the degree of $u$. Just as $G$, the graph $G'$ must have exactly $N_1/2$ edges. The number of edges between $N(u)$ and $N'(u)$ may not change by swapping edges and non-edges so it must be exactly $k(N-1-k)/2$.

Denote $T:=Z_{\frac 1 a, x_0}$ and, similarly to the proof of Theorem \ref{thm:bound2}, consider two-distance derived sets $T_1=T_{\frac 1 {a+1}, x}$ and $T_2=T_{\frac {-1} {a-1}, x}$. $T_1$ is a two-distance set with $k$ points and scalar products $\frac 1 {a+2}$ and $-\frac {a+3}{(a-1)(a+2)}$. Denote by $t_1$ and $t_2$ the numbers of ordered pairs of points in $T_1$ such that their scalar products are $\frac 1 {a+2}$ and $-\frac {a+3}{(a-1)(a+2)}$ respectively. $T_2$ is a two-distance set with $N-k-1$ points and scalar products $-\frac 1 {a-2}$ and $\frac {a-3}{(a+1)(a-2)}$. Denote by $t'_1$ and $t'_2$ the numbers of ordered pairs of points in $T_2$ such that their scalar products are $\frac {a-3}{(a+1)(a-2)}$ and $-\frac 1 {a-2}$ respectively. We know that $N_1=2k+t_1+t'_1+k(N-k-1)$ (the degree of $u$ is counted twice and the number of pairs between $N(u)$ and $N'(u)$ is exactly $k(N-1-k)$ as we showed above).

Both for $T_1$ and $T_2$ we can find lower bounds on $t_1$ and $t'_1$ using positive semidefiniteness of their Gram matrices. For $T_1$ we have

$$t_1+t_2=k(k-1)\ \ \text{ and }\ \ k+\frac 1 {a+2} t_1 + -\frac {a+3}{(a-1)(a+2)} t_2\geq 0.$$

Therefore,

$$t_1\geq \frac {k^2(a+3)-k(a+1)^2}{2(a+1)}.$$

Similarly, for $T_2$ we have

$$t'_1+t'_2=(N-k-1)(N-k-2)\ \ \text{ and }\ \ (N-k-1)+\frac {a-3}{(a+1)(a-2)} t'_1 + (-\frac 1 {a-2})t'_2\geq 0.$$

Therefore,

$$t'_1\geq\frac{(a+1)(N-k-1)(N-k-a)}{2(a-1)}.$$

From these two inequalities and equality $N_1=\frac {(a+1)N(N-a)}{2a}$ we get that

$$\frac {(a+1)N(N-a)}{2a}\geq 2k +\frac {k^2(a+3)-k(a+1)^2}{2(a+1)} + \frac{(a+1)(N-k-1)(N-k-a)}{2(a-1)} + k(N-k-1).$$

Transforming this inequality, we obtain

$$0\geq \frac {2a}{a^2-1}\left(k-\frac {(N-a)(a+1)}{2a}\right)^2.$$

We conclude that $k$ must be $\frac {(N-a)(a+1)}{2a}$. The choice of point $x$ was arbitrary which means that all vertices in $G$ must have the same degree $k$ and, moreover, all degrees in $G'$ after swapping edges and non-edges between $N(u)$ and $N'(u)$ must be still the same. This means that from each vertex of $N'(u)$ there are exactly $N(u)/2=k/2$ edges to $N(u)$, i.e. any two disconnected vertices of $G$ have exactly $k/2$ common neighbors. Similarly, for any vertex of $N(u)$ there are exactly $N'(u)/2=(N-k-1)/2$ edges to $N'(u)$ which leaves $k-1-(N-k-1)/2=(3k-N-1)/2$ edges to other vertices of $N(u)$. This means that any two connected vertices of $G$ must have $(3k-N-1)/2$ common neighbors. Therefore, $G$ is a strongly regular graph with parameters $(N,k,\frac {3k-N-1}{2},\frac k 2)$.

For the remaining part of the proof we employ the fact that adjacency matrices $\Phi$ of graph $G$ and $J-I-\Phi$ of its complement have the same spectral structure. One of the eigenspaces is the one-dimensional space generated by {\bf 1} and the other two eigenspaces have dimensions we can find as described above. Using the notation from the beginning of this section, we find that $e_{1}=\frac {N-a}{2a}$ with multiplicity $d_1=\frac {N(a^2-1)}{N+a^2}$ and $e_2=-\frac {a+1}{2}$ with multiplicity $d_2=\frac {N^2-a^2}{N+a^2}$.

The Gram matrix of $Z$ is $I+\frac 1 {a+1}\Phi+\frac {-1}{a-1} (J-I-\Phi)$ and, as the benefit of the spectral structure of $\Phi$, we can find all eigenvalues of this matrix with their multiplicities:

$$1+\frac 1 {a+1} k+\frac {-1}{a-1}\left(N-k-1\right)=1+\frac 1 {a+1} \left(-\frac {a+1}{2}\right)-\frac {1}{a-1} \left(-1-\frac {a+1}{2}\right)=0$$

so 0 is an eigenvalue with multiplicity $d_2+1$;

$$1+\frac 1 {a+1} \frac {N-a}{2a}-\frac 1 {a-1} \left(-1-\frac {N-a}{2a}\right)=\frac {N+a^2}{a^2-1}$$

so this is an eigenvalue of multiplicity $d_1=\frac {N(a^2-1)}{N+a^2}$. The rank of the Gram matrix is then $\frac {N(a^2-1)}{N+a^2}=a^2-3$ so $Z$ belongs to the $(a^2-3)$-dimensional space. Therefore, the initial equiangular set $X$ is $(a^2-2)$-dimensional.
\end{proof}

Generally, Theorem \ref{thm:extr1} claims that, given $n\leq 3a^2 -16$, the Gerzon-type bound of Theorem \ref{thm:equiang} is attainable if and only if the actual Gerzon bound is attainable and, moreover, the extremal set is precisely the set where the Gerzon bound is attained. This may be used to slightly improve upper bounds for the maximum size of an equiangular set in some dimensions. For example, it is known \cite{ban04} that there is no equiangular set in $\mathbb{R}^{47}$ with 1128 points. Due to Theorem \ref{thm:extr1}, all equiangular sets with 1128 points in dimensions 48-75 must be embeddable in $\mathbb{R}^{47}$ and, therefore, such sets don't exist either. This means that $M(n)\leq 1127$ for all $n$ from 47 to 75. Similarly, $M(n)\leq 3159$ for $n\in[79,116]$ and $M(n)\leq 14027$ for $n\in[167, 221]$.

\begin{remark}
When proving that $G$ is strongly regular, we didn't use that $N$ is precisely $\frac{(a^2-3)a^2}{2}$ so this part essentially gives a new self-contained proof for the classification of equiangular tight frames and their connection to strongly regular graphs with parameters $(N,k,\frac {3k-N-1}{2},\frac k 2)$ (see, for instance, \cite{wal09}). For the classification of all spherical two-distance tight frames see \cite{bar15}.
\end{remark}

\extremaltwo*

\begin{proof}
As in the proof of Theorem \ref{thm:equiang}, assume that the bound is attained on the set with scalar products $\frac 1 b$ and $-\frac 1 b$. Of the three cases in the proof, only in the first it is possible to reach the bound. Hence $b^2-2\leq n\leq 3b^2-16$. Moreover, since $M_{\frac 1 b} (n)\leq \frac{(b^2-2)(b^2-1)}{2}$, the bound is attained only if $b=a$, which means that this set must have scalar products $\frac 1 a$ and $-\frac 1 a$. The second part of the corollary statement automatically follows from Theorem \ref{thm:extr1}.
\end{proof}

\section{Discussion}\label{sect:discuss}

In this section we would like to list several open questions and our conjectures as well as some general directions for research in this area.

\begin{enumerate}

\item We proved that maximum spherical $n$-dimensional two-distance sets, bar a set of exceptional dimensions, have $\frac{n(n+1)} 2$ points. However, the proof does not give a characterization of all maximal sets. The methods we employ can give a lot of information about two-distance sets with scalar products $\alpha+\beta\leq 0$. Unfortunately, the only result for the case $\alpha+\beta>0$ is Musin's lemma (Lemma \ref{lem:musin}) which doesn't provide any insight into extremal cases. We conjecture that for almost all dimensions (the density of exceptional dimensions from 1 to $N$ converges to 0 when $N\rightarrow \infty$) the extremal configurations are natural configurations of midpoints of the regular simplex. Moreover, we think that, for a fixed natural $c$, all two-distance spherical configurations with at least $\frac{n(n+1)} 2-c$ points are subsets of these natural configurations for almost all dimensions.

\item We conjecture that in each of the exceptional cases, $n=(2k+1)^2-3$, if there is no spherical two-distance set with $\frac{n(n+3)} 2$ vertices, then the maximal set has $\frac{n(n+1)} 2$ vertices. In other words, there are no intermediate two-distance sets: the extremal set is either a very special set bearing a lot of structural properties (tight spherical design, strongly regular graph, etc) or is simply the set of midpoints of the regular simplex.

\item The harmonic bound (Theorem \ref{thm:s-dist}) may not be attained for any $s\geq 3$. This follows from the non-existence of tight spherical $2s$-designes for $s\geq 3$ \cite{ban79}. However, no general bounds asymptotically improving the harmonic bound are known. We conjecture that the natural configuration with $\binom{n+1}{s}$ points is the maximal $s$-distance configuration for almost all $n$.

\item Due to the construction of de Caen \cite{cae00} and Gerzon's bound,

$$\frac 2 9 \leq \limsup_{n\rightarrow\infty} \frac {M(n)}{n^2}\leq \frac 1 2.$$

The exact value of this upper limit is an open question.

As one of the outcomes of this paper we know that there are no equiangular sets with scalar product $\frac 1 a$ of size $\sim n^2$ if $a^2\ll n$. We can adjust the question of finding the upper limit to the case of all $a$ and, using \cite{bal16} and Theorem \ref{thm:bound2}, find similar bounds for $M_{\frac 1 a}(n)$:

$$\frac 1 3 \leq \limsup_{n\rightarrow\infty} \max_{a} \frac {M_{\frac 1 a}(n)}{n a^2}\leq \frac 2 3.$$

The lower bound here is also obtained from de Caen's construction. It will be interesting to find this number exactly or improve existing bounds.

\item Theorem \ref{thm:extr1} and Corollary \ref{cor:extr2} show that in cases where the Gerzon-type bound is attained it is always given by a set in dimension $(2k+1)^2-2$. For each $n$, it seems reasonable to analyze equiangular sets of dimensions precisely $n$ (not smaller) and define $M^*(n)$ as the size of a maximal equiangular set of dimension precisely $n$ (not embeddable in $\mathbb{R}^{n-1}$). One of the consequences of Theorem \ref{thm:extr1} is that $M^*(n)$ is not monotonous. For example, $M^*(21)<M^*(22)$ and $M^*(23)<M^*(22)$. It will be very interesting to find lower bounds on $M^*(n)$. So far it is not even clear that $M^*(n)$ has a quadratic lower bound.

\end{enumerate}

\section{Acknowledgments}

We thank Alexander Barg for his invaluable comments on the text of the paper. Alexey Glazyrin was supported in part by NSF grant DMS-1400876.

\bibliographystyle{amsalpha}

\end{document}